\documentstyle[amssymb,amsfonts,12pt]{amsart}

\newtheorem{theorem}{Theorem}[section]
\newtheorem{lemma}[theorem]{Lemma}
\newtheorem{corollary}[theorem]{Corollary}
\newtheorem{proposition}[theorem]{Proposition}
\newtheorem{definition}[theorem]{Definition}

\newtheorem{conjecture}[theorem]{Conjecture}\newtheorem{Ansatz}[theorem]{Assumption}

\newtheorem{remark}[theorem]{Remark}
\newtheorem*{ack*}{Acknowledgment}

\def\v{{\bf v}}\def\p{\vec{p}}

\def\F{{\mathcal F}}
\def\R{{\mathbb R}}

\def\N{{\mathbb N}}
\def\C{{\mathbb C}}
\def\Q{{\mathcal T}}\def\B{{\mathcal B}}

\def\M{{\mathcal M}}\def\G{{\bf Gr}}

\def\A{{\mathcal N}}\def\SC{\mathcal C}

\def\dim{{\operatorname{dim}}}\def\Col{{\textbf{Col}}}\def\deg{{\operatorname{deg}}}

\def\supp{{\operatorname{supp}}}\def\det{{\operatorname{det}}}
\def\bas{\begin{align*}}
\def\eas{\end{align*}}
\def\bi{\begin{itemize}}
\def\ei{\end{itemize}}
\newenvironment{proof}{\noindent {\bf Proof} }{\endprf\par}
\def \endprf{\hfill  {\vrule height6pt width6pt depth0pt}\medskip}
\def\emph#1{{\it #1}}

\begin{document}
\author{Jean Bourgain}
\address{School of Mathematics, Institute for Advanced Study, Princeton, NJ 08540}
\email{bourgain@@math.ias.edu}
\author{Ciprian Demeter}
\address{Department of Mathematics, Indiana University, 831 East 3rd St., Bloomington IN 47405}
\email{demeterc@@indiana.edu}
\thanks{The first author is partially supported by the NSF grant DMS-1301619. The second  author is partially supported  by the NSF Grant DMS-1161752}
\title[Mean value estimates for  Weyl sums]{Mean value estimates for  Weyl sums in two dimensions}

\begin{abstract}
We use decoupling theory to estimate the number of solutions  for quadratic and cubic Parsell--Vinogradov systems in two dimensions.
\end{abstract}
\maketitle

\section{Introduction}
\medskip
For $k\ge 2$ let $\M_{2,k}$ be the two dimensional manifold in $\R^n=\R^{\frac{k(k+3)}2}$
\begin{equation}
\label{ctuv45t8vt9590v0rt98h]owpq,dirfx[z]qwpxe]}
\M_{2,k}=\{(t,s,\Psi(t,s)):(t,s)\in[0,1]^2\},
\end{equation}
where the entries of $\Psi(t,s)$ consist of all the monomials $t^is^j$ with $2\le i+j\le k$.

For each square $R\subset [0,1]^2$ and each $g:R\to\C$ define the extension operator associated with $\M_{2,k}$
\begin{equation}
\label{lp[.gi0ty-hl9yb./v045tiuh,89}
E_R^{(k)}g(x_1,\ldots,x_n)=\int_Rg(t,s)e(x_1t+x_2s+x_3t^2+x_4s^2+x_5st+\ldots)dtds.
\end{equation}
In particular,
$$E_R^{(2)}g(x_1,\ldots,x_5)=\int_Rg(t,s)e(x_1t+x_2s+x_3t^2+x_4s^2+x_5st)dtds,$$
$$E_R^{(3)}g(x_1,\ldots,x_9)=\int_Rg(t,s)e(x_1t+x_2s+x_3t^2+x_4s^2+x_5st+x_6t^3+x_7s^3+x_8t^2s+x_9ts^2)dtds.$$

Here and throughout the rest of the paper we will  write $$e(z)=e^{2\pi i z},\;z\in\R.$$
For a positive weight $v:\R^n\to[0,\infty)$ we define
$$\|f\|_{L^p(v)}=(\int_{\R^n}|f(x)|^pv(x)dx)^{1/p}.$$
Also, for each ball $B$ in $\R^n$ centered at $c(B)$ and with radius $R$, $w_{B}$ will denote the weight
$$w_{B}(x)= \frac{1}{(1+\frac{|x-c(B)|}{R})^{100n}}.$$

For $N\ge 1$ and $p\ge 2$, let $D_k(N,p)$ be the smallest constant such that
$$\|E_{[0,1]^2}^{(k)}g\|_{L^p(w_{B_{N}})}
\le D_k(N,p)(\sum_{\Delta\subset [0,1]^2\atop{l(\Delta)=N^{-1/k}}}\|E_\Delta^{(k)} g\|_{L^p(w_{B_{N}})}^p)^{1/p},
$$
for each $g:[0,1]^2\to\C$ and each ball $B_N\subset\R^n$ with radius $N$, where the sum  is over a finitely overlapping cover of $[0,1]^2$ with squares $\Delta$ of side length $l(\Delta)=N^{-1/k}$.
\bigskip

Our main result is the following decoupling theorem for $\M_{2,k}$, when $k\in\{2,3\}$.
\begin{theorem}
\label{tfek6}

\begin{enumerate}

\item ($k=2$) For each $p\ge 2$ we have
$$D_2(N,p)\lesssim_{\epsilon,p} N^{\frac12-\frac1p+\epsilon},\,2\le p\le 8,$$
$$D_2(N,p)\lesssim_{\epsilon,p} N^{1-\frac5{p}+\epsilon},\,p\ge 8.$$

\item ($k=3$) For each $2\le p\le 16$
\begin{equation}
\label{fe3added-later}
D_3(N,p)\lesssim_{\epsilon,p} N^{\frac{2}3(\frac12-\frac1p)+\epsilon},\,2\le p\le 16.
\end{equation}
\end{enumerate}
\end{theorem}

Led by the number theoretical considerations from Section \ref{numbthe} (see also the computation in Section 6 from \cite{BD4}), it seems reasonable to conjecture the following result.
\begin{conjecture} For each $k\ge 2$ we have
$$D_k(N,p)\lesssim_{\epsilon,p} N^{\frac2{k}(\frac12-\frac1p)+\epsilon},\,\,\,2\le p\le \frac{k(k+1)(k+2)}{3}.$$
\end{conjecture}
Here $N^{\frac2k}$ is the number of squares with side length $N^{-1/k}$ in a finitely overlapping cover of $[0,1]^2$. Note that we prove this conjecture when $k=2$, but when $k=3$, our estimate at $p=16$ falls short of the conjectured $p=20$ threshold. The methods in this paper also prove the above conjecture for $2\le p\le k(k+3)-2$ when $k\ge 4$, conditional to Conjecture \ref{ckconjh} (see Section \ref{Trans}) which involves purely linear algebra considerations.

For future use, we record the following trivial upper bound that follows from the Cauchy--Schwartz inequality
\begin{equation}
\label{fe3008}
D_k(N,p)\lesssim N^{\frac2k(1-\frac1p)},\text{ for } p\ge 1,\;k\ge 2.
\end{equation}
\bigskip

Theorem \ref{tfek6} is part of a program that has been initiated by the authors in \cite{BD3},  where the sharp decoupling theory has been completed for hyper-surfaces with definite second fundamental form, and also for the cone. The decoupling theory has since proved to be a very successful tool for a wide variety of problems in number theory that involve exponential sums. See  \cite{Bo}, \cite{Bo6}, \cite{BW}, \cite{BD4}, \cite{BD5}. This paper is no exception from the rule. Theorem \ref{tfek6} is in part motivated by its application to Parsell--Vinogradov  systems in two dimensions, as explained in the next section. Perhaps surprisingly, our Fourier analytic approach eliminates any appeal to number theory.

Our method also  allows to replace $\M_{2,k}$ with certain perturbed versions, making it suitable for other potential applications. This perspective of exploiting the decoupling theory for more exotic manifolds has led to new estimates on the Riemann zeta function in \cite{Bo}, \cite{BW}. See also the second part of Section \ref{numbthe} here for another application.

\bigskip

Theorem \ref{tfek6} can be seen as a generalization to two dimensions of our Theorem 1.4 from \cite{BD4}, which addresses the case $d=1$ (curves). As a result, the proof here will follow a strategy similar to the one from \cite{BD4}. At the heart of the argument lies the interplay between linear and multilinear decoupling, facilitated by the Bourgain--Guth induction on scales. Running this machinery produces two types of contributions, a transverse one and a non-transverse one. To control the transverse term we need to prove a multilinear restriction theorem for a specific two dimensional manifold in $\R^n$. Defining transversality in a manner that makes it easy to check and achieve in our application, turns out to be a rather delicate manner. A novelty in the current setting is that the non-transverse contribution comes from  neighborhoods of  zero sets of a polynomial functions  $Q(t,s)$ of degree greater than one. This forces us to work with a family of multilinear estimates, rather than just one.

In the attempt to simplify the discussion, we often run non-quantitative arguments that  rely instead on compactness. For example, in line with our previous related papers, we never care about the exact quantitative dependence on transversality of the bound in the multilinear restriction inequality. These considerations occupy sections \ref{Bra}, \ref{Trans} and \ref{se:multi}.

The key multi-scale inequality is presented in Section \ref{Maininewsecthgg}. We have decided to present it in a greater generality, to make it easily available for potential forthcoming applications.

\begin{ack*}
We thank Trevor Wooley for a few stimulating discussions and to  Jonathan Bennett for sharing the manuscript \cite{BBFL}, which plays a crucial role in the proof of our Theorem \ref{tfek2}. We thank the referee for a careful reading of the original manuscript and for making a few suggestions which led to the simplification of the arguments.  The second author  would like to thank Mariusz Mirek and Lillian Pierce for drawing his attention to the Vinogradov mean value theorem in higher dimensions.
\end{ack*}

\bigskip

\section{Number theoretical consequences}
\label{numbthe}
Here we  present two applications of Theorem \ref{tfek6}.
\subsection{Parsell--Vinogradov systems}
For each integer $s\ge 1$, denote by  $J_{s,2,2}(N)$ the number of integral solutions for the following quadratic Parsell--Vinogradov  system
$$X_1+\ldots+X_s=X_{s+1}+\ldots+X_{2s},$$
$$Y_1+\ldots+Y_s=Y_{s+1}+\ldots+Y_{2s},$$
$$X_1^2+\ldots+X_s^2=X_{s+1}^2+\ldots+X_{2s}^2,$$
$$Y_1^2+\ldots+Y_s^2=Y_{s+1}^2+\ldots+Y_{2s}^2,$$
$$X_1Y_1+\ldots+X_sY_s=X_{s+1}Y_{s+1}+\ldots+X_{2s}Y_{2s},$$
with $1\le X_i,Y_j\le N$. Note that this system is naturally associated with the manifold $\M_{2,2}$. By adding four more equations which are cubic in the variables $X_i,Y_j$ one gets a system associated with $\M_{2,3}$. A similar construction works for all $\M_{2,k}$, $k\ge 2$, and following \cite{PPW}, the corresponding number of solutions is denoted by $J_{s,k,2}$.

We will restrict attention to $k=2,3$. It was  conjectured in \cite{PPW} (see the top of page 1965) that for $s\ge 1$
\begin{equation}
\label{lhuij80iok-0809ol-=0o8=l-./p-p}
J_{s,2,2}(N)\lesssim_{\epsilon,s} N^{\epsilon}(N^{2s}+N^{4s-8}),
\end{equation}
and
\begin{equation}
\label{lhuij80iok-0809ol-=0o8=l-./p-p1}
J_{s,3,2}(N)\lesssim_{\epsilon,s} N^{\epsilon}(N^{2s}+N^{4s-20}).
\end{equation}

Theorem 1.1 in \cite{PPW} established \eqref{lhuij80iok-0809ol-=0o8=l-./p-p} for $s\ge 15$ and \eqref{lhuij80iok-0809ol-=0o8=l-./p-p1} for $s\ge 36$. Here we will prove the following two estimates.
\begin{theorem}
Inequality \eqref{lhuij80iok-0809ol-=0o8=l-./p-p}  holds  in the whole range $s\ge 1$. Inequality \eqref{lhuij80iok-0809ol-=0o8=l-./p-p1} holds for $1\le s\le 8$.
\end{theorem}

Trevor Wooley has pointed out to us that there is an alternative proof for \eqref{lhuij80iok-0809ol-=0o8=l-./p-p} at the critical exponent $s=5$, using the Siegel mass formula. This type of argument does not work for $k\ge 3$.
When $k\ge 4$, our argument gives the expected estimate for $J_{s,k,2}$ when $1\le s\le \frac{k(k+3)}{2}-1$, conditional to Conjecture \ref{ckconjh}. This range is rather poor for large values of $k$ and this did not justify putting  any serious effort into proving Conjecture \ref{ckconjh} for $k\ge 4$.
\bigskip

To simplify numerology and notation, we prove the above theorem when $k=2$. The case $k=3$ is treated very similarly.

Our approach will in fact prove a much more general result, see Corollary \ref{cfek4} below. We start with the following discrete restriction estimate which follows quite easily from our Theorem \ref{tfek6}.

\begin{theorem}
For each $1\le i\le N$, let $t_i,s_i$ be two points in $(\frac{i-1}{N},\frac{i}{N}]$. Then for each $R\gtrsim N^{2}\ge 1$, each ball $B_R$ with radius $R$ in $\R^5$, each $a_{i,j}\in\C$ and each $p\ge 2$  we have
$$(\frac1{|B_R|}\int_{B_R}|\sum_{i=1}^N\sum_{j=1}^Na_{i,j}e(x_1s_i+x_2t_j+x_3s_i^2+x_4t_j^2+x_5s_it_j)|^{p}dx_1\ldots dx_5)^{\frac1p}\lesssim$$
\begin{equation}
\label{fek19}
 D_2(N^2,p)\|a_{i,j}\|_{l^p(\{1,\ldots,N\}^2)},
\end{equation}
and the implicit constant does not depend on $N$, $R$ and $a_{i,j}$.
\end{theorem}
\begin{proof}
Given $B_R$, let $\B$ be a finitely overlapping cover of $B_R$ with balls $B_{N^2}$. An elementary computation shows that
\begin{equation}
\label{fek20}
\sum_{B_{N^2}\in\B}w_{B_{N^2}}\lesssim w_{B_R},
\end{equation}
with the implicit constant independent of $N,R$. Invoking  Theorem \ref{tfek6} for each $B_{N^2}\in\B$, then summing up and using
\eqref{fek20} we obtain
$$
\|E_{[0,1]^2}g\|_{L^p(B_R)}\lesssim $$
$$
 D_2(N^2,p)(\sum_{\Delta\subset [0,1]^2\atop{l(\Delta)=N^{-1}}}\|E_\Delta g\|_{L^p(w_{B_{R}})}^p)^{1/p}.
$$
Use this inequality with $$g=\frac1{\tau^2}\sum_{i=1}^N\sum_{j=1}^Na_{i,j}1_{B_{i,j,\tau}},$$
where $B_{i,j,\tau}$ is the ball in $\R^2$ centered at $(s_i,t_j)$ with radius $\tau.$ Then let $\tau$ go to $0$.

\end{proof}
For each $1\le i\le N$ consider some real  numbers $i-1< \tilde{X}_i,\tilde{Y}_i\le i$. We do not insist that $\tilde{X}_i,\tilde{Y}_i$ be integers. Let $S_X=\{\tilde{X}_1,\ldots,\tilde{X}_N\}$ and $S_Y=\{\tilde{Y}_1,\ldots,\tilde{Y}_N\}$. For each $s\ge 1$, denote by  $\tilde{J}_{s,2,2}(S_X,S_Y)$ the number of  solutions  of the following system of inequalities
$$|X_1+\ldots+X_s-(X_{s+1}+\ldots+X_{2s})|\le \frac1{N},$$
$$|Y_1+\ldots+Y_s-(Y_{s+1}+\ldots+Y_{2s})|\le \frac1{N},$$
$$|X_1^2+\ldots+X_s^2-(X_{s+1}^2+\ldots+X_{2s}^2)|\le 1,$$
$$|Y_1^2+\ldots+Y_s^2-(Y_{s+1}^2+\ldots+Y_{2s}^2)|\le 1,$$
$$|X_1Y_1+\ldots+X_sY_s-(X_{s+1}Y_{s+1}+\ldots+X_{2s}Y_{2s})|\le 1,$$
with $X_i\in S_X, Y_j\in S_Y$.
\begin{corollary}
\label{cfek4}
For each integer $s\ge 1$ and each $S_X,S_Y$ as above we have that
$$\tilde{J}_{s,2,2}(S_X,S_Y)\lesssim_{\epsilon,s}N^{\epsilon}(N^{2s}+N^{4s-8}),$$
where the implicit constant does not depend on $S_X,S_Y$.
\end{corollary}
\begin{proof}Let $\phi:\R^5\to [0,\infty)$ be a positive Schwartz function with  positive Fourier transform  satisfying $\widehat{\phi}(\xi)\ge1$ for $|\xi|\lesssim 1$.
Define $\phi_{N}(x)=\phi(\frac{x}N)$. Using the Schwartz decay, \eqref{fek19} with $a_{i,j}=1$ implies that for each $s\ge 1$
$$(\frac1{|B_{N^2}|}\int_{\R^5}\phi_{N^2}(x_1,\ldots,x_5)|\sum_{i=1}^N\sum_{j=1}^Ne(x_1s_i+x_2t_j+x_3s_i^2+x_4t_j^2+x_5s_it_j)|^{2s}dx_1\ldots dx_5)^{\frac1{2s}}\lesssim$$
\begin{equation}
\label{fek21}
 D_2(N^2,2s)N^{\frac1s},
\end{equation}
whenever $s_i,t_i\in [\frac{i-1}{N},\frac{i}{N})$. Apply \eqref{fek21} to $s_i=\frac{\tilde{X}_i}{N}$ and $t_j=\frac{\tilde{Y}_j}{N}$.
Let now $$\phi_{N, 1}(x_1,\ldots,x_5)=\phi(\frac{x_1}N,\frac{x_2}N, {x_3}, {x_4},{x_5}).$$

After making a change of variables and expanding the product,  the term
$$\int_{\R^5}\phi_{N^2}(x_1,\ldots,x_5)|\sum_{i=1}^N\sum_{j=1}^Ne(x_1s_i+x_2t_j+x_3s_i^2+x_4t_j^2+x_5s_it_j)|^{2s}dx_1\ldots dx_5$$
can be written as  the sum over all $X_i\in S_X,Y_j\in S_Y$ of
$$N^8\int_{\R^5}\phi_{N, 1}(x_1,\ldots,x_5)e(x_1Z_1+x_2Z_2+x_3Z_3+x_4Z_4+x_5Z_5)dx_1\ldots dx_5,$$
where
$$Z_1=X_1+\ldots+X_s-(X_{s+1}+\ldots+X_{2s}),$$
$$Z_2=Y_1+\ldots+Y_s-(Y_{s+1}+\ldots+Y_{2s}),$$
$$Z_3=X_1^2+\ldots+X_s^2-(X_{s+1}^2+\ldots+X_{2s}^2),$$
$$Z_4=Y_1^2+\ldots+Y_s^2-(Y_{s+1}^2+\ldots+Y_{2s}^2),$$
$$Z_5=X_1Y_1+\ldots+X_sY_s-(X_{s+1}Y_{s+1}+\ldots+X_{2s}Y_{2s}).$$
Each such term is equal to
$$N^{10}\widehat{\phi}(NZ_1,NZ_2,Z_3,Z_4, Z_5).$$
Recall that this is always positive, and in fact greater than $N^{10}$ at least $\tilde{J}_{s,2,2}(S_X,S_Y)$ times. Going back to \eqref{fek21}, it follows by invoking Theorem \ref{tfek6} that
$$\tilde{J}_{s,2,2}(S_X,S_Y)\lesssim D_2(N^2,2s)N^{\frac1s}\lesssim_{\epsilon,s} N^{\epsilon}(N^{2s}+N^{4s-8}).$$

\end{proof}
\subsection{Perturbed manifolds}
Theorem \ref{tfek6} with $k=2$ remains true if $\M_{2,2}$ is replaced with
$$\M=\{(t,s,\Psi^1(t,s),\Psi^2(t,s),\Psi^3(t,s))\},$$
assuming the real-valued $\Psi^i$ satisfy the non-degeneracy condition
$$\det \begin{bmatrix}\Psi^{1}_{tt}&\Psi^{1}_{ss}&\Psi^{1}_{ts}\\ \Psi^{2}_{tt}&\Psi^{2}_{ss}&\Psi^{2}_{ts}\\ \Psi^{3}_{tt}&\Psi^{3}_{ss}&\Psi^{3}_{ts}\end{bmatrix}(t,s)\not=0.$$
We refer the reader to \cite{BD5} for the details on a related scenario. Applying this to
$$(\Psi^1(t,s),\Psi^2(t,s),\Psi^3(t,s))=(t^4,s^4,t^2s^2)$$
for $t,s\sim 1$ we can prove the following result.
\begin{corollary}
The system of inequalities
$$X_1+\ldots+X_4=X_{5}+\ldots+X_{8},$$
$$Y_1+\ldots+Y_4=Y_{5}+\ldots+Y_8,$$
$$|X_1^4+\ldots+X_4^4-(X_{5}^4+\ldots+X_{8}^4)|\lesssim N^2,$$
$$|Y_1^4+\ldots+Y_4^4-(Y_{5}^4+\ldots+Y_{8}^4)|\lesssim N^2,$$
$$|X_1^2Y_1^2+\ldots+X_4^2Y_4^2-(X_{5}^2Y_{5}^2+\ldots+X_{8}^2Y_{8}^2)|\lesssim N^2,$$
has $O_\epsilon(N^{8+\epsilon})$ integral solutions $X_i,Y_j\sim N$.
\end{corollary}
To understand the numerology, note that there are $\sim N^8$ trivial solutions. The proof follows considerations similar to those in the previous subsection. See also the proof of Theorem 2.18 in \cite{BD3}.

\bigskip

\section{A Brascamp--Lieb inequality}
\label{Bra}

For $1\le j\le m$, let $V_j$ be $n_j-$dimensional affine subspaces of $\R^n$ and let $l_j:\R^n\to V_j$ be surjective affine transformations.  Define the multilinear functional
$$\Lambda(f_1,\ldots,f_m)=\int_{\R^n}\prod_{j=1}^mf_j(l_j(x))dx$$
for $f_j:V_j\to\C$. Each $V_j$ will be equipped with the $n_j-$dimensional Lebesgue measure. We recall the following theorem from \cite{BCCT}.
\begin{theorem}
\label{BCCT}
Given a vector $\p=(p_1,\ldots,p_m)$ with $p_j\ge 1$, we have that
\begin{equation}
\label{fek17}
\sup_{f_j\in L^{p_j}({V_j})}\frac{|\Lambda(f_1,\ldots,f_m)|}{\prod_{j=1}^m\|f_j\|_{L^{p_j}}}<\infty
\end{equation}
if and only if
\begin{equation}
\label{fek1}
n=\sum_{j=1}^m\frac{n_j}{p_j}
\end{equation}
and the following transversality condition is satisfied
\begin{equation}
\label{fek2}
\dim(V)\le \sum_{j=1}^m\frac{\dim(l_j(V))}{p_j}, \text{ for every subspace }V\subset \R^n.
\end{equation}
\end{theorem}
When all $p_j$ are equal to some $p$, an equivalent way to write \eqref{fek17} is
\begin{equation}
\label{fek18}
\sup_{g_j\in L^{2}({V_j})}\frac{\|(\prod_{j=1}^m g_j\circ l_j)^{\frac1m}\|_{L^{q}}}{(\prod_{j=1}^m\|g_j\|_{L^{2}})^{\frac1m}}<\infty,
\end{equation}
where $q=\frac{2nm}{\sum_{j=1}^m n_j}$.

We will be interested in the special case when $V_j$ are linear subspaces, $l_j=\pi_j$ are orthogonal projections and $n_j=2$.   For future use, we reformulate the theorem in this case.
\begin{theorem}
\label{BCCT1}
The quantity
$$\sup_{g_j\in L^{2}({V_j})}\frac{\|(\prod_{j=1}^{m}(g_j\circ \pi_j))^{\frac1{m}}\|_{L^n(\R^n)}}{(\prod_{j=1}^{m}\|g_j\|_{L^{2}(V_j)})^{\frac1{m}}}$$
is finite if and only if
\begin{equation}
\label{fek3}
\dim(V)\le \frac{n}{2m}\sum_{j=1}^{m}{\dim(\pi_j(V))}, \text{ for every linear subspace }V\subset \R^n.
\end{equation}
\end{theorem}

Remark \ref{rek456} will show the relevance of the space $L^n$ from Theorem \ref{BCCT1}.

\bigskip

\section{Transversality}
\label{Trans}
For $k\ge 2$ recall the definition \eqref{ctuv45t8vt9590v0rt98h]owpq,dirfx[z]qwpxe]} of the two dimensional manifold $\M_{2,k}$ in $\R^n=\R^{\frac{k(k+3)}2}$.
Denote by
$$n_1(t,s)=(1,0,\ldots)$$
$$n_2(t,s)=(0,1,\ldots)$$
the canonical tangent vectors to $\M_{2,k}$ at $(t,s,\Psi(t,s))$.

In this section we  introduce a quantitative form of transversality for $\M_{2,k}$ as well as Conjecture \ref{ckconjh}, which we prove for $k=2,3$. The key result in this section, that we prove conditional to Conjecture \ref{ckconjh} is  a multilinear Kakeya-type inequality. This will then lead to the proof of the multilinear restriction Theorem \ref{tfek4} in Section \ref{se:multi}.

Given two vectors $v_1, v_2$ in $\R^{\frac{k(k+3)}2}$, define the polynomial function on $\R^2$
$$Q_{v_1,v_2}(t,s)=\det \begin{bmatrix}n_1(t,s)\cdot v_1&n_1(t,s)\cdot v_2\\n_2(t,s)\cdot v_1&n_2(t,s)\cdot v_2\end{bmatrix}.$$Note that its degree is at most $2k-2$. The following lemma will be relevant for our discussion of the cases $k=2$ and $k=3$.
\medskip

\begin{lemma}\label{lBnew}
(a) $(k=2)$ There does not exist a three dimensional space $V$ in $\R^5$ so that $Q_{v,w}\equiv 0$ for each $v,w\in V$.
\medskip

(b) $(k=3)$ There does not exist a five dimensional space $V$ in $\R^9$ so that $Q_{v,w}\equiv 0$ for each $v,w\in V$.
\end{lemma}
\begin{proof}
We first prove (a). Assume for contradiction such a $V$ exists. The requirement $Q_{v,w}\equiv 0$ means
\begin{equation}
\label{yyyyyyyyyy7}
\det \begin{bmatrix}v_1+2v_3t+v_5s&w_1+2w_3t+w_5s\\v_2+2v_4s+v_5t&w_2+2w_4s+w_5t\end{bmatrix}=0
\end{equation}
for each $t,s$. In particular, by taking into account only the coefficients of the terms of order $\le 1$ we get
\begin{equation}
\label{yyyyyyyyyy4}
v_1w_2-v_2w_1=0,
\end{equation}
\begin{equation}
\label{yyyyyyyyyy5}
v_1w_5-v_5w_1+2v_3w_2-2v_2w_3=0,
\end{equation}
\begin{equation}
\label{yyyyyyyyyy6}
v_5w_2-v_2w_5+2v_1w_4-2v_4w_1=0.
\end{equation}
Adding these up we get
$$v_1(w_2+w_5+2w_4)+v_2(-w_1-2w_3-w_5)+2v_3w_2-2v_4w_1+v_5(-w_1+w_2)=0.$$
Since this holds for all $v,w\in V$, this leads to the inclusion
$$T_1V\subset V^{\perp},$$
where $$T_1(w)=(w_2+w_5+2w_4,-w_1-2w_3-w_5,2w_2,-2w_1,-w_1+w_2).$$
Since $3=\dim(V)>\dim(V^\perp)=2$, by the Rank-Nullity Theorem, it follows that the kernel of $T_1$ restricted to $V$ must be nontrivial. But the kernel of $T_1$ on $\R^5$ is the one dimensional space spanned by $(0,0,1,1,-2)$, which forces $(0,0,1,1,-2)\in V$. Using this in \eqref{yyyyyyyyyy7} we find that for each $w\in V$
$$(t-s)(w_1+w_2+t(2w_3+w_5)+s(2w_4+w_5))\equiv 0,$$
or
$$w_1+w_2=2w_3+w_5=2w_4+w_5=0.$$
This shows that $\dim(V)\le 2$, leading to a contradiction.
\bigskip

The proof of (b) is very similar. Assume for contradiction that such a $V$ exists.  The requirement $Q_{v,w}\equiv 0$ means
\begin{equation}
\label{yyyyyyyyyy8}
\det \begin{bmatrix}v_1+2v_3t+v_5s+3v_6t^2+2v_8st+v_9s^2&w_1+2w_3t+w_5s+3w_6t^2+2w_8st+w_9s^2\\v_2+2v_4s+v_5t+3v_7s^2+v_8t^2+2v_9st&w_2+2w_4s+w_5t+3w_7s^2+w_8t^2+2w_9st\end{bmatrix}=0
\end{equation}
for each $t,s$. By taking into account  the coefficients of the terms of order $\le 1$ we get that \eqref{yyyyyyyyyy4}, \eqref{yyyyyyyyyy5} and \eqref{yyyyyyyyyy6} continue to hold in this case, too. Moreover, by considering the coefficients of $t^4$ and $s^4$ we also get
\begin{equation}
\label{yyyyyyyyyy9}
v_6w_8-v_8w_6=0,
\end{equation}
\begin{equation}
\label{yyyyyyyyyy10}
v_7w_9-v_9w_7=0.
\end{equation}
Adding up \eqref{yyyyyyyyyy4}, \eqref{yyyyyyyyyy5}, \eqref{yyyyyyyyyy6}, \eqref{yyyyyyyyyy9} and \eqref{yyyyyyyyyy10} we get
$$T_2V\subset V^{\perp},$$
where
$$T_2(w)=(w_2+w_5+2w_4,-w_1-2w_3-w_5,2w_2,-2w_1,-w_1+w_2,w_8,w_9,-w_6,-w_7).$$
A similar argument as before finishes the proof, once we notice that the kernel of $T_2$ on $\R^9$ is the one dimensional space spanned by $(0,0,1,1,-2,0,0,0,0)$.
\end{proof}

We will denote by $[x]$ the integer part of $x$. It seems plausible to conjecture the following extension to higher dimensions.
\begin{conjecture}
\label{ckconjh}
Let $n=\frac{k(k+3)}{2}$, $k\ge 2$. Then there does not exist a $[\frac{n-1}2]+1-$dimensional space $V$ in $\R^n$ such that $Q_{v,w}\equiv 0$ for each $v,w\in V$.
\end{conjecture}

\bigskip

Given a polynomial function $Q(t,s)$ of any degree $\deg(Q)$, denote by $\|Q\|$ the $l^2$ norm of its coefficients.
\begin{definition}
\label{dfek1}
Let $n=\frac{k(k+3)}{2}$.
A collection consisting of $m\ge n$ sets $S_1,\ldots,S_{m}\subset [0,1]^2$ is said to be $\nu-$transverse for  $\M_{2,k}$ if the following  requirement is satisfied:
\medskip

For each $1\le i_1\not= i_2\ldots\not= i_{\left[\frac{m}{n}\right]+1}\le m$  we have
\begin{equation}
\label{fek4}
\inf_{\deg(Q)\le 2k-2,\atop{\|Q\|=1}}\max_{1\le j\le \left[\frac{m}{n}\right]+1}\inf_{(t,s)\in S_{i_j}}|Q(t,s)|\ge \nu.
\end{equation}

\end{definition}
 Note that transverse sets are not necessarily pairwise disjoint. Requirement \eqref{fek4} says that $\left[\frac{m}{n}\right]+1$ points in  different sets $S_i$ do not come "close`` to belonging to the zero set of a polynomial function $Q$ of degree $\le 2k-2$. This is a rather weak form of transversality, but it is easily seen to have the two attributes that we need. First, large enough collections of squares will contain a transverse subcollection, as shown in Theorem \ref{tfek3}. Second, transverse squares will satisfy the requirement needed for the application of the Brascamp--Lieb inequality, as shown in the following result.

\begin{proposition}
\label{pfek1}

Assume Conjecture \ref{ckconjh} holds for some $k\ge 2$. Consider $m\ge n$ points $(t_j,s_j)\in [0,1]^2$ such that the sets $S_j=\{(t_j,s_j)\}$ are $\nu-$transverse for $\M_{2,k}$, for some $\nu>0$. Then the $m$ planes $V_j,\;1\le j\le m$  spanned by the vectors $n_1(t_j,s_j)$ and $n_2(t_j,s_j)$ in $\R^n$ satisfy requirement \eqref{fek3}.

\end{proposition}
\begin{proof}
It suffices to check \eqref{fek3} for linear subspaces $V$ with dimension between one and $n-1$, as the case of dimension zero or $n$ is trivial.
\medskip

Note that given any nonzero vector $v$,  at least one of  $v\cdot n_1(t,s)=0$ and $v\cdot n_2(t,s)=0$ represents a nontrivial polynomial function $Q$ with degree $\le 2k-2$. The first observation is that a one dimensional subspace can not be orthogonal to $\left[\frac{m}{n}\right]+1$ distinct $V_j$. If this were to be the case, the $\left[\frac{m}{n}\right]+1$ planes $V_j$ would be forced to belong to a hyperplane in $\R^n$, with normal vector $v$. But then the $\left[\frac{m}{n}\right]+1$ corresponding points $(t_j,s_j)$ would belong to both  $v\cdot n_1(t,s)=0$ and  $v\cdot n_2(t,s)=0$, contradicting \eqref{fek4}. This observation shows that \eqref{fek3} is satisfied if $\dim(V)\le \left[\frac{n-1}{2}\right]$, as $\dim(\pi_j(V))\ge 1$ for at least $m-\left[\frac{m}{n}\right]\ge \frac{m(n-1)}{n}$ values of $j$.
\medskip

Consider now the case of $V$ with $\left[\frac{n-1}{2}\right]+1\le \dim(V)\le n-1$. Let $V'$ be an arbitrary subspace of $V$ with $\dim(V')=\left[\frac{n-1}{2}\right]+1$ and basis $v_1,\dots,v_{\left[\frac{n-1}{2}\right]+1}$. We will argue that there can be at most  $\left[\frac{m}{n}\right]$ planes $V_j$ with $\dim(\pi_j(V))\le 1$.  This immediately implies \eqref{fek3}, as
$$\frac{n}{2m}\sum_{j=1}^m\dim(\pi_j(V))\ge 2\frac{n}{2m}(m-\left[\frac{m}{n}\right])\ge n-1\ge \dim(V).$$
Assume now for contradiction that $\dim(\pi_j(V))\le 1$ for $\left[\frac{m}{n}\right]+1$ values of $j$, that is to say $1\le j\le \left[\frac{m}{n}\right]+1$. Obviously $\dim(\pi_j(V'))\le 1$, too. By the Rank-Nullity Theorem, the rank of the matrix
$$\begin{bmatrix}n_1(t_j,s_j)\cdot v_1&\dots& n_1(t_j,s_j)\cdot v_{\left[\frac{n-1}{2}\right]+1}\\ n_2(t_j,s_j)\cdot v_{1}&\ldots& n_2(t_j,s_j)\cdot v_{\left[\frac{n-1}{2}\right]+1}\end{bmatrix}$$
is at most one. In particular
$$Q_{u_1,u_2}(t_j,s_j)=0,\;1\le j\le \left[\frac{m}{n}\right]+1$$
for each $u_1,u_2\in V'$.
Using Conjecture \ref{ckconjh}, we can pick $u_1,u_2\in V'$ so that $Q_{u_1,u_2}$ is nontrivial, and this contradicts \eqref{fek4}, as $\deg(Q_{u_1,u_2})\le 2k-2$.

\end{proof}

A $K-$square will be a  closed square in $[0,1]^2$ with side length $\frac1K$. When $K=2^l$ for $l\in\N$, the collection of all dyadic $K-$squares will be denoted by $\Col_K$. Since $\Col_K$ is finite, the various constants throughout the rest of the argument can be made uniform over the choice of squares.

The relevance of the following simple result will be clear in the proof of Proposition \ref{fp2}.
\begin{theorem}
\label{tfek3}
There exists $\Lambda=\Lambda_k>0$ such that for each $K\ge 1$ there exists $\nu_{K}=\nu_{K,k}>0$   so that any $\Lambda K$ or more squares in $\Col_K$  are $\nu_{K}-$transverse for $\M_{2,k}$.
\end{theorem}
\begin{proof}
Let $d\ge 1$. By the main theorem in \cite{Wo} it follows that the $\frac{10}{K}-$neighborhood in $[0,1]^2$ of the zero set of any polynomial of degree $\le d$ in two variables will intersect at most $C_dK$ squares in $\Col_K$. The quantity
$$\nu_K:=\min_{\Col\subset\Col_K\;\;\atop{|\Col|\ge (C_{2k-2}+1)K}}\inf_{\deg(Q)\le {2k-2},\atop{\|Q\|=1}}\max_{R\in \Col}\inf_{(t,s)\in R}|Q(t,s)|$$
is easily seen to be positive, via a compactness argument. We can take $\Lambda_k=(C_{2k-2}+1)$
\end{proof}
\bigskip

For $k\ge 2$ let $n=\frac{k(k+3)}{2}$ and $\Lambda=\Lambda_k$. Let\; ${\SC}_{K}={\SC}_{K,k}$ denote the collection of all  $\Lambda K-$tuples $(V_1,\ldots,V_{\Lambda K})$ of  planes spanned by the vectors $n_1(t_j,s_j)$, $n_2(t_j,s_j)$ in $\R^{n}$ with $(t_j,s_j)$ arbitrary points belonging to distinct squares\footnote{A point can of course belong to as many as four squares. We only ask for the existence of a choice of distinct squares to which the points belong} $R_j\in\Col_K$.

Given $R\in\Col_K$, the collection $\Q_R(\delta)$ consists of the $\delta$ neighborhoods $T
$ of planes parallel to the plane spanned by $n_1(t,s)$, $n_2(t,s)$, for some arbitrary $(t,s)\in R$.

We can now prove the following multilinear Kakeya-type inequality.

\begin{theorem}
\label{tfek2}
Assume Conjecture \ref{ckconjh} holds for some $k\ge 2$.
Then  there exists a constant $\Theta_K<\infty$ depending  on $K$ so that for each $0<\delta<1$, for each  pairwise distinct $R_j\in \Col_K$ and for each finite subsets $\Q_{R_j}'(\delta)\subset \Q_{R_j}(\delta)$ we have
$$\|(\prod_{j=1}^{\Lambda K}(\sum_{T\in \Q_{R_j}'(\delta)}1_T))^{1/\Lambda K}\|_{L^{n/2}([-1,1]^n)}\lesssim_\epsilon \Theta_K \delta^{2-\epsilon}(\prod_{j=1}^{\Lambda K}|\Q_{R_j}'(\delta)|)^{1/\Lambda K}.$$
\end{theorem}
\begin{proof}
The proof will rely on a few well-known  observations, as well as on a recent result from \cite{BBFL}.
The Grassmannian $\G(2,\R^n)$ is the collection of all (two dimensional) planes containing the origin in $\R^n$. It is a compact metric space when equipped with the metric
$$d_{\G(2,\R^n)}(X,Y)=\|P_X-P_Y\|,$$
where $P_X,P_Y$ are the associated projections, and their difference is measured in the operator norm. Consider the function $$F:\G(2,\R^n)^{\Lambda K}\to\C^*$$ defined by
$$F(V_1,\ldots,V_{\Lambda K})=\sup_{g_j\in L^{2}({V_j})}\frac{\|(\prod_{j=1}^{\Lambda K}(g_j\circ \pi_j))^{\frac1{\Lambda K}}\|_{L^n(\R^n)}}{(\prod_{j=1}^{\Lambda K}\|g_j\|_{L^{2}(V_j)})^{\frac1{\Lambda K}}}.$$

Theorems \ref{BCCT1} and \ref{tfek3} combined with Proposition \ref{pfek1} show that $F(V_1,\ldots,V_{\Lambda K})<\infty$ if $(V_1,\ldots,V_{\Lambda K})\in{\SC}_K$.

Theorem 1.2 in \cite{BBFL} proves that if $F(V_1,\ldots,V_{\Lambda K})<\infty$ then there exist $\Theta_{(V_1,\ldots,V_{\Lambda K})}<\infty$ and $\nu_{(V_1,\ldots,V_{\Lambda K})}$ so that the inequality
$$\|(\prod_{j=1}^{\Lambda K}(\sum_{T\in \Q_{j}'(\delta)}1_T))^{1/\Lambda K}\|_{L^{n/2}([-1,1]^n)}\lesssim_\epsilon \Theta_{(V_1,\ldots,V_{\Lambda K})} \delta^{2-\epsilon}(\prod_{j=1}^{\Lambda K}|\Q_{j}'(\delta)|)^{1/\Lambda K}$$
holds for each finite collections $\Q_{j}'(\delta)$ consisting of $\delta$ neighborhoods of planes $V_j'$, so that, up to translation,  $(V_1',\ldots,V_{\Lambda K}')$ is within distance $\nu_{(V_1,\ldots,V_{\Lambda K})}$ from $(V_1,\ldots,V_{\Lambda K})$ in $\G(2,\R^n)$.

It is rather immediate that ${\SC}_K$  is closed in $\G(2,\R^n)^{\Lambda K}$ (each square $R\in\Col_K$ is closed), hence compact. The previous observation produces an open cover of ${\SC}_K$, which will necessarily contain a finite subcover. The theorem now follows.

\end{proof}
\bigskip

\section{The multilinear restriction theorem}
\label{se:multi}
Recall the manifold $\M_{2,k}$ from \eqref{ctuv45t8vt9590v0rt98h]owpq,dirfx[z]qwpxe]}.
For each $S\subset [0,1]^2$ and each  $0<\delta<1$, let $\A_{S,\delta}=\A_{k, S,\delta}$ be the $\delta-$neighborhood of
$$\M_{k,S}:=\{(t,s,\Psi(t,s))\in \M_{2,k}:(t,s)\in S\}.$$

The key result recorded in this section is the  multilinear restriction Theorem \ref{tfek4}. This is a close relative of the multilinear restriction theorem of Bennett, Carbery and Tao \cite{BCT}, which has been recently generalized in \cite{BBFL} by Bennett, Bez, Flock and Lee.
Recall the definition of $\Lambda=\Lambda_k$ from Theorem \ref{tfek3}.

\begin{theorem}
\label{tfek4}Assume Conjecture \ref{ckconjh} holds for some $k\ge 2$. Then, for each pairwise distinct  squares $R_1,\ldots,R_{\Lambda K}\in \Col_K$, each $f_j:\A_{ R_j,\frac1N}\to\C$, each $\epsilon>0$ and  each ball $B_N$ in $\R^n=\R^{\frac{k(k+3)}{2}}$ with radius $N\ge 1$  we have
\begin{equation}
\label{rt ngbrtbpvopfirty54ifitutp[reripf[t5mu89}
\|(\prod_{j=1}^{\Lambda K}\widehat{f}_j)^{\frac1{\Lambda K}}\|_{L^n(B_N)}\lesssim_{\epsilon,K} N^{\epsilon-\frac{n-2}{2}}(\prod_{j=1}^{\Lambda K}\|f_j\|_{L^2(\A_{R_j,\frac1N})})^{\frac1{\Lambda K}}.
\end{equation}
\end{theorem}
The implicit constant in \eqref{rt ngbrtbpvopfirty54ifitutp[reripf[t5mu89} will depend on $\epsilon$ and on the quantity $\Theta_K$ from Theorem \ref{tfek2}. The derivation of this theorem from the related multilinear Kakeya-type inequality follows the argument from  \cite{BCT}, \cite{BBFL}. We omit the details and refer the interested reader to these papers.

\section{The main inequality}
\label{Maininewsecthgg}
In this section we present the key result to our induction on scales-based approach, Proposition \ref{iovjurgyptn8vbgu89357893v7589ty7056893}. To make this strategy easily accessible for future applications,  we chose to present it in a greater generality, allowing for arbitrary manifolds of arbitrary dimension $d$. Our approach is somewhat abstract. There will be no explicit mention or use of transversality, but this will be implicitly contained in Assumption \ref{Ansatz1}.

Let $\M$ be a manifold in $\R^n$ which is the graph $(u,\Psi(u))$ of a $C^\infty$ function $\Psi:[0,1]^d\to\R^{n-d}$.
For each $S\subset [0,1]^d$ and each $\delta>0$, let $\A_{S,\delta}$ be the $\delta-$neighborhood of
$$\M_S:=\{(u,\Psi(u)):\;u\in S\}.$$
Define also the extension operator associated with $\M$
$$E_Sg(x)=\int_Sg(u)e(u\cdot x^d+\Psi(u)\cdot\bar{x})du,\;\;x=(x^d,\bar{x})\in \R^{d}\times \R^{n-d}.$$

Throughout the remainder of the section we fix $\M$, as well as the positive integer $m\ge 1$, the $d$ dimensional squares $R_1,\ldots,R_m\subset [0,1]^d$ and the real number $r\ge 2$,  and we assume that the following holds.
\medskip
\begin{Ansatz}
\label{Ansatz1}
For each $f_j$ supported on $\A_{R_j,\frac1N}$, each $\epsilon>0$ and  each ball $B_N$ in $\R^n$ with radius $N\ge 1$  we have
\begin{equation}
\label{omjngumrgu745r-03th564ro0tu}
\|(\prod_{j=1}^{m}\widehat{f}_j)^{\frac1{m}}\|_{L^r(B_N)}\le \Gamma(\epsilon) N^{\epsilon-\frac{n-d}2}(\prod_{j=1}^{m}\|f_j\|_{L^2(\A_{R_j,\frac1N})})^{\frac1{m}},
\end{equation}
with $\Gamma(\epsilon)=\Gamma(R_1,\ldots,R_m,\epsilon)$ depending on $\epsilon$ but not on $N$, $B_N$ and $f_j$.
\end{Ansatz}

\begin{remark}
\label{rek456}
It is rather immediate that
$$\|(\prod_{j=1}^{m}\widehat{f}_j)^{\frac1{m}}\|_{L^{\infty}(B_N)}\le (\prod_{j=1}^{m}\|\widehat{f}_j\|_{L^{\infty}(\R^n)})^{\frac1{m}}\lesssim  N^{-\frac{n-d}{2}}(\prod_{j=1}^{m}\|f_j\|_{L^2(\A_{R_j,\frac1N})})^{\frac1{m}}.$$
Thus, if \eqref{omjngumrgu745r-03th564ro0tu} holds, then in fact
$$
\|(\prod_{j=1}^{m}\widehat{f}_j)^{\frac1{m}}\|_{L^p(B_N)}\lesssim_{\Gamma(\epsilon),p} N^{\epsilon-\frac{n-d}{2}}(\prod_{j=1}^{m}\|f_j\|_{L^2(\A_{ R_j,\frac1N})})^{\frac1{m}}
$$
holds for each $p\ge r$.

It is worth observing that \eqref{omjngumrgu745r-03th564ro0tu} can not hold for $r<\frac{2n}{d}$ . Indeed, use
$\widehat{f_j}=\phi_{T_j}$, where $\phi_{T_j}$ is a single wave-packet. We can arrange  for the intersection of the plates $T_j$ to contain a ball of radius $\sim N^{\frac12}$. Then \eqref{omjngumrgu745r-03th564ro0tu} yields
$$N^{\frac{n}{2r}}\lesssim_{\Gamma(\epsilon)}N^{\epsilon-\frac{n-d}2}(N^{\frac{d}2}N^{n-d})^{1/2},$$
which amounts to $r\ge \frac{2n}{d}$.

 Typically, \eqref{omjngumrgu745r-03th564ro0tu} can be ensured to be true under  appropriate transversality requirements for the sets $R_i$. These requirements are fairly easy to state when $d=n-1$ and $d=1$, and have led to multilinear theorems in \cite{BCT} and \cite{BD4}, for $r=\frac{2n}{d}$. In Section \ref{se:multi} we proved a similar result when $d=2$, namely Theorem \ref{tfek4}, for the particular manifold $\M_{2,k}$, $k=2,3$.

\end{remark}
\bigskip

We will now derive various consequences of Assumption \ref{Ansatz1}. It is important to realize that all implicit constants will depend on $\Gamma(\epsilon)$.
We start with the following reformulation, which we will prefer in our applications.
\begin{theorem}
\label{tfek1}
If Assumption \ref{Ansatz1} holds true, then for each $g_j:R_j\to\C$, each ball $B_N\subset\R^n$ with radius $N\ge 1$ and each $\epsilon>0$ we have
$$\|(\prod_{j=1}^{m}E_{R_j}g_j)^{\frac1{m}}\|_{L^r(B_N)}\lesssim_{\Gamma(\epsilon)} N^{\epsilon}(\prod_{j=1}^{m}\|g_j\|_{L^2(R_j)})^{\frac1{m}}.$$
\end{theorem}

To see that  Assumption \ref{Ansatz1} implies  Theorem \ref{tfek1}, choose a positive Schwartz function $\eta$ on $\R^n$ such that
$$1_{B(0,1)}\le \eta,\;\;\text{ and }\;\;\supp\;\widehat{\eta}\subset B(0,\frac1{100}),$$
and let
\begin{equation}
\label{fek15}
\eta_{B_N}(x)=\eta(\frac{x-c(B_N)}{N}).
\end{equation}
Then, for $g_j$ as in Theorem \ref{tfek1},
$$\|(\prod_{j=1}^{m}E_{R_j}g_j)^{\frac1{m}}\|_{L^r(B_N)}\le \|(\prod_{j=1}^{m}((E_{R_j}g_j)\eta_{B_N}))^{\frac1{m}}\|_{L^r(B_N)}.$$
It suffices to note that the Fourier transform of $(E_{R_j}g_j)\eta_{B_N}$ is supported in $\A_{R_j,\frac1N}$ and that its $L^2$ norm is  $O(N^{-\frac{n-d}2}\|g_j\|_2)$.
\bigskip

We have the following consequence of Theorem \ref{tfek1}.
\begin{corollary}
\label{fc1}
If Assumption \ref{Ansatz1} holds true, then  for each $r\le p\le \infty$, $\epsilon>0$, each ball $B_N\subset\R^n$ with radius $N\ge 1$ and  $g_i:R_i\to\C$ we have
\begin{equation}
\label{we3}
\|(\prod_{i=1}^{m}\sum_{\atop{l(\Delta)=N^{-1/2}}} |E_{\Delta}g_i|^2)^{\frac1{2m}}\|_{L^{p}(w_{B_N})}\lesssim_{\Gamma(\epsilon),p} N^{-\frac{r(n-d)}{2p}+\epsilon}(\prod_{i=1}^{m}\sum_{\atop{l(\Delta)=N^{-1/2}}}\|E_{\Delta}g_i\|_{L^{\frac{2p}r}(w_{B_N})}^2)^{\frac1{2m}}.
\end{equation}
\end{corollary}
\begin{proof}
Consider the function $\eta_{B_N}$ introduced earlier. Note that for each $i$ the functions $((E_{\Delta}g_i)\eta_{B_N})_{\Delta}$ are almost orthogonal in the $L^2$ sense. Combining this  with  Assumption \ref{Ansatz1} we get the following local inequality
$$\|(\prod_{i=1}^{m} |E_{R_i}g_i|)^{1/m}\|_{L^{r}(w_{B_N})}\lesssim_{\Gamma(\epsilon)} N^{-\frac{n-d}2+\epsilon}(\prod_{i=1}^{m}\sum_{\atop{l(\Delta)=N^{-1/2}}}\|E_{\Delta}g_i\|_{L^{2}(w_{B_N})}^2)^{\frac1{2m}}.$$
A  randomization argument further leads to the inequality
$$\|(\prod_{i=1}^{m}\sum_{\atop{l(\Delta)=N^{-1/2}}} |E_{\Delta}g_i|^2)^{\frac1{2m}}\|_{L^{r}(w_{B_N})}\lesssim_{\Gamma(\epsilon)} N^{-\frac{n-d}2+\epsilon}(\prod_{i=1}^{m}\sum_{\atop{l(\Delta)=N^{-1/2}}}\|E_{\Delta}g_i\|_{L^{2}(w_{B_N})}^2)^{\frac1{2m}}.$$
It now suffices to interpolate this with the trivial inequality
$$\|(\prod_{i=1}^{m}\sum_{\atop{l(\Delta)=N^{-1/2}}} |E_{\Delta}g_i|^2)^{\frac1{2m}}\|_{L^{\infty}(w_{B_N})}\le (\prod_{i=1}^{m}\sum_{\atop{l(\Delta)=N^{-1/2}}}\|E_{\Delta}g_i\|_{L^{\infty}(w_{B_N})}^2)^{\frac1{2m}}.$$
We refer the reader to \cite{BD3} for how this type of interpolation is performed.
\end{proof}
\bigskip

For $p\ge r$ define $\kappa_p$ such that
$$\frac{r}{2p}=\frac{1-\kappa_p}{2}+\frac{\kappa_p}{p},$$
in other words, $$\kappa_p=\frac{p-r}{p-2}.$$

As observed earlier, the case $r=\frac{2n}{d}$ is an endpoint, so it will naturally  produce the strongest applications. In this case, we get the following key inequality.

\begin{proposition}
\label{iovjurgyptn8vbgu89357893v7589ty7056893}
If Assumption \ref{Ansatz1} holds true with
$$r=\frac{2n}{d},$$then
for  each ball $B_R$ in $\R^n$ with radius $R\ge N\ge 1$, $p\ge r$, $\epsilon>0$, $\kappa_p\le \kappa\le 1$ and $g_i:R_i\to \C$ we have
$$\|(\prod_{i=1}^{m}\sum_{\atop{l(\tau)=N^{-1/2}}}|E_{\tau}g_i|^2)^{\frac1{2m}}\|_{L^{p}(w_{B_R})}\lesssim_{\Gamma(\epsilon),p}$$
$$
N^{\epsilon}\|(\prod_{i=1}^{m}\sum_{\atop{l(\Delta)=N^{-1}}}|E_{\Delta}g_i|^2)^{\frac1{2m}}
\|_{L^{p}(w_{B_R})}^{1-\kappa}
(\prod_{i=1}^{m}\sum_{\atop{l(\tau)=N^{-1/2}}}\|E_{\tau}g_i\|_{L^{p}(w_{B_R})}^2)^{\frac{\kappa}{2m}}.
$$
\end{proposition}
\bigskip

\begin{remark}
A simple computation using $g_i=1_{\tau_i}$, with $\tau_i$ an arbitrary $d$ dimensional square with $l(\tau)=N^{-1/2}$, shows that the inequality is false for $\kappa<\kappa_p$. This is the main restriction that prevents us from getting a better range in Theorem \ref{tfek6}, as will become apparent throughout the computations done in the last section of the paper.
\end{remark}
\bigskip

\begin{proof}
The inequality is immediate for $\kappa=1$ via a combination of H\"older's and Minkowski's inequalities. It thus suffices to prove it for $\kappa=\kappa_p$.

Let  $B$ be an arbitrary ball of radius $N$ in $\R^n$. We start by recalling that \eqref{we3} on $B$ gives
\begin{equation}
\label{fe20}
\|(\prod_{i=1}^{m}\sum_{\atop{l(\tau)=N^{-1/2}}} |E_{\tau}g_i|^2)^{\frac1{2m}}\|_{L^{p}(w_{B})}\lesssim_{\Gamma(\epsilon),p} N^{-\frac{r(n-d)}{2p}+\epsilon}(\prod_{i=1}^{m}\sum_{\atop{l(\tau)=N^{-1/2}}}\|E_{\tau}g_i\|_{L^{2p/r}(w_{B})}^2)^{\frac1{2m}}.
\end{equation}
Write using H\"older's inequality
\begin{equation}
\label{we7}
(\sum_{\atop{l(\tau)=N^{-1/2}}}\|E_{\tau}g_i\|_{L^{2p/r}(w_{B})}^2)^{\frac1{2}}\le (\sum_{\atop{l(\tau)=N^{-1/2}}}\|E_{\tau}g_i\|_{L^{2}(w_{B})}^2)^{\frac{1-\kappa_p}{2}}(\sum_{\atop{l(\tau)=N^{-1/2}}}\|E_{\tau}g_i\|_{L^{p}(w_{B})}^2)^{\frac{\kappa_p}{2}}.
\end{equation}
The  key element in our argument is the almost orthogonality specific to $L^2$, which will allow us to pass from scale $N^{-1/2}$ to scale $N^{-1}$. Indeed, since $(E_\Delta g_i)w_{B}$ are almost orthogonal for $l(\Delta)=N^{-1}$, we have
\begin{equation}
\label{dskjuhu9vhyvg89tyvgn894}
(\sum_{\atop{l(\tau)=N^{-1/2}}}\|E_{\tau}g_i\|_{L^{2}(w_{B})}^2)^{1/2}\lesssim (\sum_{\atop{l(\Delta)=N^{-1}}}\|E_{\Delta}g_i\|_{L^{2}(w_{B})}^2)^{1/2}.
\end{equation}
We can now rely on the fact that $|E_{\Delta}g_i|$ is essentially constant on balls $B'$ of radius $N$ to argue that
$$(\sum_{\atop{l(\Delta)=N^{-1}}}\|E_{\Delta}g_i\|_{L^{2}(B')}^2)^{\frac1{2}}\sim |B'|^{1/2}(\sum_{\atop{l(\Delta)=N^{-1}}}|E_{\Delta}g_i(x)|^2)^{\frac1{2}},\text{ for }x\in{B'}$$
and thus
\begin{equation}
\label{we8}
(\prod_{i=1}^{m}\sum_{\atop{l(\Delta)=N^{-1}}}\|E_{\Delta}g_i\|_{L^{2}(w_{B})}^2)^{\frac1{2m}}\lesssim |B|^{\frac12-\frac1p}\|(\prod_{i=1}^{m}\sum_{\atop{l(\Delta)=N^{-1}}}|E_{\Delta}g_i|^2)^{\frac1{2m}}\|_{L^{p}(w_{B})}.
\end{equation}
Combining \eqref{fe20},  \eqref{we7}, \eqref{dskjuhu9vhyvg89tyvgn894}, \eqref{we8} with the fact that
$$n(\frac12-\frac1p)\frac{r-2}{p-2}=r\frac{n-d}{2p},\;\;\text{when }r=\frac{2n}{d},$$ we get
$$
\|(\prod_{i=1}^{m}\sum_{\atop{l(\tau')=N^{-1/2}}}|E_{\tau}g_i|^2)^{\frac1{2m}}\|_{L^{p}(w_{B})}\lesssim_{\Gamma(\epsilon),p}$$
$$N^{\epsilon} \|(\prod_{i=1}^{m}\sum_{\atop{l(\Delta)=N^{-1}}}|E_{\Delta}g_i|^2)^{\frac1{2m}}\|_{L^{p}(w_{B})}^{1-\kappa_p}
(\prod_{i=1}^{m}\sum_{\atop{l(\tau)=N^{-1/2}}}\|E_{\tau}g_i\|_{L^{p}(w_{B})}^2)^{\frac{\kappa_p}{2m}}.
$$
Summing this up over a finitely overlapping family of balls $B\subset B_R$ of radius $N$, we get the desired inequality, upon invoking the inequalities of H\"older and Minkowski.
\end{proof}
\medskip

We close this section with specializing the result of Proposition \ref{iovjurgyptn8vbgu89357893v7589ty7056893} to the manifold $\M=\M_{2,k}$. Recall the notation for the extension operator $E^{(k)}$ defined by $\M_{2,k}$.

\begin{corollary}
\label{gjityiophjytophpotigirti0-we=fdcwee=w=}
Assume Conjecture \ref{ckconjh} holds for some $k\ge 2$ and let $n=\frac{k(k+3)}{2}$. Then for each $K\ge 2$, $p\ge n$ and $\epsilon>0$ there exists a constant $C_{p,K,\epsilon}$ such that for each pairwise distinct  squares $R_1,\ldots,R_{\Lambda K}\in \Col_K$, each ball $B_R$ in $\R^n$ with radius $R\ge N\ge 1$ and each $g_i:R_i\to \C$ we have
$$\|(\prod_{i=1}^{\Lambda K}\sum_{\atop{l(\tau)=N^{-1/2}}}|E_{\tau}^{(k)}g_i|^2)^{\frac1{2\Lambda K}}\|_{L^{p}(w_{B_R})}\le $$$$C_{p,K,\epsilon}
N^{\epsilon}\|(\prod_{i=1}^{\Lambda K}\sum_{\atop{l(\Delta)=N^{-1}}}|E_{\Delta}^{(k)}g_i|^2)^{\frac1{2\Lambda K}}
\|_{L^{p}(w_{B_R})}^{1-\kappa_p}
(\prod_{i=1}^{\Lambda K}\sum_{\atop{l(\tau)=N^{-1/2}}}\|E_{\tau}^{(k)}g_i\|_{L^{p}(w_{B_R})}^2)^{\frac{\kappa_p}{2\Lambda K}},
$$
where $$\kappa_p=\frac{p-n}{p-2}.$$
\end{corollary}
\begin{proof}
Theorem \ref{tfek4} shows  that Assumption \ref{Ansatz1} is satisfied with $\M=\M_{2,k}$, $m=K\Lambda$ and $r=n$, for each pairwise distinct  squares $R_1,\ldots,R_{\Lambda K}\in \Col_K$. Moreover, the constant $\Gamma(\epsilon)$ in \eqref{omjngumrgu745r-03th564ro0tu} will depend on $K$.
\end{proof}

\section{Rescaling}
A crucial feature of our argument is the fact that the manifold $\M_{2,k}$  has a certain invariance under rescaling. Recall the definition \eqref{lp[.gi0ty-hl9yb./v045tiuh,89} of  the extension operator $E^{(k)}$ defined by the manifold $\M_{2,k}$.

Assume $R=[a,a+\delta]\times [b,b+\delta]$. The affine change of variables
$$(t,s)\in R\mapsto(t',s')=\eta(t,s)=(\frac{t-a}{\delta},\frac{s-b}{\delta})\in[0,1]^2$$ shows that
$$|E_R^{(k)}g(x)|=\delta^2|E_{[0,1]^2}^{(k)}g^{a,b}(\bar{x})|$$
where
$$g^{a,b}(t',s')=g(\delta t'+a,\delta s'+b),$$
and the relation between $x=(x_1,\ldots,x_n)$ and $\bar{x}=(\bar{x}_1,\ldots,\bar{x}_n)$ is
$$\bar{x}_1=\delta(x_1+2ax_3+bx_5+\ldots),$$
$$\bar{x}_2=\delta(x_2+2bx_4+ax_5+\ldots),$$
$$\bar{x}_3=\delta^2(x_3+\ldots),\;\;\bar{x}_4=\delta^2(x_4+\ldots),\ldots,\bar{x}_n=\delta^k(x_n+\ldots). $$
One of the key applications of this invariance is given by the following result.
\begin{proposition}
\label{rvty090-0w=qw}
For each $p\ge 1$, each square $R=[a,a+\delta]\times [b,b+\delta]$ with side length $\delta=N^{-\rho}$,  $\rho\le \frac1k$ and each ball $B_N$ in $\R^n=\R^{\frac{k(k+3)}{2}}$ we have
\begin{equation}
\label{fe22}
\|E_R^{(k)}g\|_{L^p(w_{B_N})}\lesssim D_k(N^{1-k\rho},p)(\sum_{\Delta\subset R\atop{l(\Delta)=N^{-1/k}}}\|E_\Delta^{(k)} g\|_{L^p(w_{B_N})}^p)^{1/p}.
\end{equation}
\end{proposition}
\begin{proof}
It suffices to prove that
$$\|E_R^{(k)}g\|_{L^p(B_N)}\lesssim D_k(N^{1-k\rho},p)(\sum_{\Delta\subset R\atop{l(\Delta)=N^{-1/k}}}\|E_\Delta^{(k)} g\|_{L^p(w_{B_N})}^p)^{1/p},$$
where the left hand side has no weight.
Note that $\bar{x}$ is the image of $x$ under a shear transformation $S$. Call $C_N=S(B_N)$ the image of the ball $B_N$  under this transformation. This is essentially a
$$\delta N\times\delta N\times\delta^2 N\times\ldots\times\delta^k N-\text{cylinder}.$$
Cover $C_N$ with a family $\F$ of balls $B_{\delta^kN}$ with $O(1)$ overlap, so that
\begin{equation}
\label{fejkhuhiofutguopfi59outo9568u}
1_{B_N}(x)\lesssim (\sum_{B_{\delta^kN}\in\F}w_{B_{\delta^kN}})(Sx)\lesssim w_{B_N}(x),\;\;x\in\R^n.
\end{equation}
After a change of variables, write
$$\|E_{R}^{(k)}g\|_{L^p(B_N)}=\delta^{2-\frac{k(k+1)(k+2)}{3p}}\|E_{[0,1]^2}^{(k)}g^{a,b}\|_{L^p(C_N)}.$$
The right hand side is bounded by
$$\delta^{2-\frac{k(k+1)(k+2)}{3p}}(\sum_{B_{\delta^kN}\in\F}\|E_{[0,1]^2}^{(k)}g^{a,b}\|_{L^p(B_{\delta^kN})}^p)^{1/p}\le$$
$$\delta^{2-\frac{k(k+1)(k+2)}{3p}}D_k(N^{1-k\rho},p)(\sum_{B_{\delta^kN}\in\F}\sum_{l(R')=\delta^{-1}N^{-\frac{1}k}}\|E_{R'}^{(k)}
g^{a,b}\|_{L^p(w_{B_{\delta^kN}})}^p)^{1/p}=$$
$$\delta^{2-\frac{k(k+1)(k+2)}{3p}}D_k(N^{1-k\rho},p)(\sum_{l(R')=\delta^{-1}N^{-\frac{1}k}}
\|E_{R'}^{(k)}g^{a,b}\|_{L^p(\sum_{B_{\delta^kN}\in\F}w_{B_{\delta^kN}})}^p)^{1/p}.$$
Changing back to the original variables and using \eqref{fejkhuhiofutguopfi59outo9568u}, we can dominate the above by
$$D_k(N^{1-k\rho},p)(\sum_{\Delta\subset R\atop{l(\Delta)=N^{-1/k}}}\|E_\Delta^{(k)} g\|_{L^p(w_{B_N})}^p)^{1/p},$$
as desired.
\end{proof}

The proof shows why one can not replace $l(\Delta)=N^{-1/k}$ with anything smaller  in \eqref{fe22}. The other application of rescaling will appear in the proof of Proposition \ref{fp3}.
\bigskip

\section{Linear versus multilinear decoupling}

Various implicit constants will be allowed to depend on the parameter $k$, but we will not record this dependence.
\bigskip

 For $2\le p<\infty$ and $N\ge 1$,  recall that $D_k(N,p)$ is the smallest constant such that the decoupling
$$\|E_{[0,1]^2}^{(k)}g\|_{L^p(w_{B_N})}\le D_k(N,p)(\sum_{l(\Delta)=N^{-1/k}}\|E_{\Delta}^{(k)}g\|_{L^p(w_{B_N})}^p)^{1/p}$$
holds true for all $g:[0,1]^2\to\C$ and all balls $B_N$ of radius $N$ in $\R^n=\R^{\frac{k(k+3)}{2}}$.
\medskip

We now introduce a family of multilinear versions of $D_k(N,p) $. Recall the definition of $\Lambda=\Lambda_k$ from Theorem \ref{tfek3}. Given  $N\ge K\ge 1$, let $D_{k,K}(N,p)$ be the smallest constant such that the inequality
$$\||\prod_{i=1}^{\Lambda K}E_{R_i}^{(k)}g_i|^{\frac1{\Lambda K}}\|_{L^p(w_{B_N})}\le D_{k,K}(N,p)(\prod_{i=1}^{\Lambda K}\sum_{l(\Delta)=N^{-1/k}}\|E_{\Delta}^{(k)}g_i\|_{L^p(w_{B_N})}^p)^{\frac1{\Lambda Kp}}$$
holds true for all distinct squares $R_1,\ldots,R_{\Lambda K}\in\Col_K$, all $g_i:R_i\to\C$  and all balls $B_N\subset\R^n$ with radius $N$.
\bigskip

Theorem \ref{tfek3} shows that for fixed $K$, any  distinct squares $R_1,\ldots,R_{\Lambda K}\in\Col_K$ are transverse for $\M_{2,k}$ in a uniform way. It is thus expected that $D_{k,K}(N,p)$ will be easier to control than $D_k(N,p)$. And indeed, as  will be seen in the proof of Corollary \ref{tiughibmuh y6m894564y078507}, the expected bound for $D_{k,K}(N,p)$ in the range $2\le p\le n$ is an immediate consequence of the multilinear  Theorem \ref{tfek4}.

H\"older's inequality shows that $D_{k,K}(N,p)\le D_k(N,p)$. The rest of the section will be devoted to proving some sort of  reverse inequality. This will  follow from a  variant of the Bourgain--Guth induction on scales in \cite{BG}. More precisely, we prove the following result.

\begin{theorem}
\label{ft2}
For each $K\ge 2$ and $p\ge 2$ there exists $\Omega_{K,p}>0$ and $\beta(K,p)>0$  with
$$\lim_{K\to \infty}\beta(K,p)=0,\;\;\text{ for each }p,$$
such that for each $N\ge K$

$$D_k(N,p)\le $$
 \begin{equation}
\label{fe31}\le N^{\beta(K,p)+\frac2k(\frac12-\frac1p)}+\Omega_{K,p}\log_K N\max_{1\le M\le N}\left[(\frac{M}{N})^{\frac2k(\frac1p-\frac12)}D_{k,K}(M,p)\right].
\end{equation}
\end{theorem}
\bigskip

Recall that we expect to have for $2\le p\le\frac{k(k+1)(k+2)}{3}$
$$D_k(N,p)\lesssim_\epsilon  N^{\frac2k(\frac12-\frac1p)+\epsilon}$$  and $$D_{k,K}(M,p)\lesssim_{\epsilon,K}  M^{\frac2k(\frac12-\frac1p)+\epsilon}.$$
Thus, if the second inequality holds, then the first one will hold, too, by invoking \eqref{fe31} and choosing $K$ large enough so that $\beta(K,p)$ is as small as desired. This relationship between $D_k(N,p)$ and $D_{k,K}(N,p)$ will be exploited in Section \ref{last}, via a  delicate bootstrapping argument.
\bigskip

The first step in the proof of Theorem \ref{ft2} is the following ``trivial" decoupling from \cite{BD5}, that we will use  to bound the non transverse contribution in the Bourgain--Guth induction on scales. For completeness, we reproduce the proof from \cite{BD5}.
\begin{lemma}
\label{fl1}
Let $R_1,\ldots,R_M$ be pairwise disjoint squares in $[0,1]^2$ with side length $K^{-1}$. Then for each $2\le p\le \infty$
$$ \|\sum_iE_{R_i}^{(k)}g\|_{L^p(w_{B_K})}\lesssim_p M^{1-\frac2p}(\sum_i\|E_{R_i}^{(k)}g\|_{L^p(w_{B_K})}^p)^{1/p}.$$
\end{lemma}
\begin{proof}
The key observation is the fact that if $f_1,\ldots,f_M:\R^n\to\C$ are such that $\widehat{f_i}$ is supported on a ball $B_i$ and the dilated balls $(2B_i)_{i=1}^M$ are pairwise disjoint, then
\begin{equation}
\label{fe36}
\|f_1+\ldots+f_M\|_{L^p(\R^n)}\lesssim_p  M^{1-\frac2p}(\sum_i\|f_i\|_{L^p(\R^n)}^p)^{\frac1p}.
\end{equation}
In fact more is true. If $T_i$ is a smooth Fourier multiplier adapted to $2B_i$ and equal to 1 on $B_i$, then the inequality
$$\|T_1(f_1)+\ldots+T_M(f_M)\|_{L^p(\R^n)}\lesssim_p  M^{1-\frac2p}(\sum_i\|f_i\|_{L^p(\R^n)}^p)^{\frac1p}$$
for arbitrary $f_i\in L^p(\R^n)$
follows by interpolating the immediate $L^2$ and $L^\infty$ estimates. Inequality \eqref{fe36} is the best one can say in general, if no further assumption is made on the Fourier supports of $f_i$. Indeed, if $\widehat{f_i}=1_{B_i}$ with $B_i$ equidistant balls of radius one with collinear centers, then the reverse inequality will hold.

Let now $\eta_{B_K}$ be as in \eqref{fek15}. It suffices to note that the Fourier supports of the  functions $f_i=\eta_{B_K}E_{R_i}^{(k)}g$ have bounded overlap.
 \end{proof}
\bigskip

The key step in proving Theorem \ref{ft2}  is the following inequality.

\begin{proposition}
\label{fp2}
For $2\le p<\infty$ and $K\ge 2$ there is a constant $C_{p}$ independent of $K$ so that for each $g:[0,1]^2\to\C$ and $N\ge K\ge 1$ we have
$$\|E_{[0,1]^2}^{(k)}g\|_{L^p(w_{B_N})}^p\le $$$$ C_{p}K^{p-2}\sum_{R\in\Col_K}\|E_{R}^{(k)}g\|_{L^p(w_{B_N})}^p+C_{p}K^{10\Lambda Kp}D_{k,K}(N,p)^p\sum_{\Delta\in
\Col_{N^{\frac1k}}}\|E_{\Delta}^{(k)}g\|_{L^p(w_{B_N})}^p.$$
\end{proposition}
\bigskip

The exponent $10\Lambda Kp$ in $K^{10\Lambda Kp}$ is not important and could easily be improved, but the exponent $p-2$ in $K^{p-2}$ is sharp and will play a critical role in the rest of the argument.
\bigskip

\begin{proof}
Following the standard formalism from \cite{BG}, we may assume that $|E_{R}^{(k)}g(x)|$ is essentially constant on each ball $B_K$ of radius $K$, and  we will denote by $c_{R,g}(B_K)$ this value. Write for each $x$
\begin{equation}
\label{mvguioyt8pbu68myi906y80-9t}
E_{[0,1]^2}^{(k)}g(x)=\sum_{R\in\Col_K}E_{R}^{(k)}g(x).
\end{equation}
Fix $B_K$. Let $R^*\in \Col_K$ be a square which maximizes the value of $c_{R,g}(B_K)$. Let $\Col_{B_K}^{*}$ be those squares  $R\in \Col_K$ such that
$$c_{R,g}(B_K)\ge K^{-2}c_{R^{*},g}(B_K).$$

We distinguish two cases.
\medskip

First, if $\Col_{B_K}^{*}$ contains at least $\Lambda K$ squares $R_1,\ldots,R_{\Lambda K}$, using \eqref{mvguioyt8pbu68myi906y80-9t} and the triangle inequality we can write
$$|E_{[0,1]^2}^{(k)}g(x)|\le K^4(\prod_{i=1}^{\Lambda K}c_{R_i,g}(B_K))^{\frac1{\Lambda K}},\;\;x\in B_K.$$

Otherwise, if $\Col_{B_K}^{*}$ contains at most $\Lambda K$ squares, we can write using the triangle inequality
$$|E_{[0,1]^2}^{(k)}g(x)|\le 2c_{R^{*},g}(B_K)+|\sum_{R\in\Col_{B_K}^{*}}E_{R}^{(k)}g(x)|,\;\;x\in B_K.$$
Next, invoking Lemma \ref{fl1} we get
$$\|E_{[0,1]^2}^{(k)}g\|_{L^p(w_{B_K})}\lesssim_p\|E_{R^*}^{(k)}g\|_{L^p(w_{B_K})}+ (\Lambda K)^{1-\frac2{p}}(\sum_{R\in\Col_{B_K}^{*}}\|E_{R}^{(k)}g\|_{L^p(w_{B_K})}^p)^{1/p}\le $$
$$\lesssim_{p}  K^{1-\frac2{p}}(\sum_{R\in\Col_{K}}\|E_{R}^{(k)}g\|_{L^p(w_{B_K})}^p)^{1/p}.$$
To summarize, in either case we can write
$$\|E_{[0,1]^2}^{(k)}g\|_{L^p(w_{B_K})}\lesssim_{p} $$$$ K^4\max_{R_1,\ldots,R_{\Lambda K}}\|(\prod_{i=1}^{\Lambda K}|E_{R_i}^{(k)}g|)^{\frac1{\Lambda K}}\|_{L^p(w_{B_K})}+ K^{1-\frac2{p}}(\sum_{R\in\Col_K}\|E_{R}^{(k)}g\|_{L^p(w_{B_K})}^p)^{1/p}\le$$
$$K^4(\sum_{R_1,\ldots,R_{\Lambda K}}\|(\prod_{i=1}^{\Lambda K}|E_{R_i}^{(k)}g|)^{\frac1{\Lambda K}}\|_{L^p(w_{B_K})}^p)^{1/p}+ K^{1-\frac2{p}}(\sum_{R\in\Col_K}\|E_{R}^{(k)}g\|_{L^p(w_{B_K})}^p)^{1/p}.$$
Raising to the power $p$  and summing over $B_K$ in a  finitely overlapping cover of $B_N$, leads to the desired conclusion.
\end{proof}
\bigskip

\bigskip

Using rescaling as in the proof of Proposition \ref{rvty090-0w=qw},  the result in Proposition \ref{fp2} leads to the following general result.
\begin{proposition}
\label{fp3}
Let $R\subset[0,1]^2$ be a square with side length $\delta$. Then for each $2\le p<\infty$, $g:R\to\C$, $K\ge 1$ and $N>\delta^{-k}K$ we have
$$\|E_{R}^{(k)}g\|_{L^p(w_{B_N})}^p\le $$$$ C_{p}K^{p-2}\sum_{R'\subset R\atop{R'\in\Col_{\frac{K}\delta}}}\|E_{R'}^{(k)}g\|_{L^p(w_{B_N})}^p+C_{p}K^{10\Lambda Kp}D_{k,K}(N\delta^{k},p)^p\sum_{\Delta\subset R\atop{\Delta\in\Col_{N^{\frac1k}}}}\|E_{\Delta}^{(k)}g\|_{L^p(w_{B_N})}^p,$$
where $C_{p}$ is the constant from Proposition \ref{fp2}.
\end{proposition}

\bigskip

We are now in position to prove Theorem \ref{ft2}. By iterating Proposition \ref{fp3} $l$ times we get
$$\|E_{[0,1]^2}^{(k)}g\|_{L^p(w_{B_N})}^p\le (C_{p}K^{p-2})^l\sum_{R\in\Col_{K^n}}\|E_{R}^{(k)}g\|_{L^p(w_{B_N})}^p+$$$$+C_{p}K^{10\Lambda Kp}\sum_{\Delta\in\Col_{N^{1/k}}}
\|E_{\Delta}^{(k)}g\|_{L^p(w_{B_N})}^p\sum_{j=0}^{l-1}(C_{p}K^{p-2})^{j}D_{k,K}(NK^{-kj},p)^p.$$
Applying this with $n$ such that $K^l=N^{\frac1k}$ we get
$$\|E_{[0,1]^2}^{(k)}g\|_{L^p(w_{B_N})}\le $$$$ N^{\frac1{kp}\log_{K}C_{p}}N^{\frac2k(\frac12-\frac1p)}(\sum_{\Delta\in\Col_{N^{1/k}}}
\|E_{\Delta}^{(k)}g\|_{L^p(w_{B_N})}^p)^{1/p}+
$$$$C_{p}^{\frac1p}K^{10 K\Lambda}\sum_{j=0}^{n-1}(\frac{NK^{-kj}}{N})^{\frac2k(\frac1p-\frac12)}
D_{k,K}(NK^{-kj},p)(\sum_{\Delta\in\Col_{N^{1/k}}}\|E_{\Delta}^{(k)}g\|_{L^p(w_{B_N})}^p)^{1/p}.
$$

The proof of Theorem \ref{ft2} is now complete, by taking $$\beta(K,p)=\frac1{kp}\log_{K}C_{p}$$
and
$$\Omega_{K,p}=\frac1kC_{p}^{1/p}K^{10K\Lambda}.$$

\bigskip
Let us now see a rather immediate application of the technology we have developed so far. Recall that $n=\frac{k(k+3)}{2}$.
\begin{corollary}
\label{tiughibmuh y6m894564y078507}
If Conjecture \ref{ckconjh} holds true for some $k\ge 2$ then
\begin{equation}
\label{ji nvhfy vguitgurtgupytooihnopyin9pnyiup97iui9076}
D_{k}(N,p)\lesssim_{\epsilon}N^{\frac2k(\frac12-\frac1{p})+\epsilon}
\end{equation}
for each $2\le p\le n$.
\end{corollary}
\begin{proof}
Using Theorem \ref{tfek4} and the fact that $E_\Delta^{(k)} g_i$ is essentially constant on $B_N$, we easily get that
$$\|(\prod_{i=1}^{\Lambda K}E_{R_i}^{(k)}g_i)^{\frac1{\Lambda K}}\|_{L^n(w_{B_N})}\lesssim_{\epsilon,K}N^{2(\frac12-\frac1n)+\epsilon}( \prod_{i=1}^{\Lambda K} (\sum_{l(\Delta)=N^{-1}}\|E_{\Delta}^{(k)}g_i\|_{L^n(w_{B_N})}^n)^{1/n})^{\frac1{\Lambda K}}$$
holds true for all pairwise distinct $R_i\in\Col_K$, each  $g_i:R_i\to\C$ and all balls $B_N$ or radius $N$ in $\R^n$.  This is a multilinear decoupling into smaller squares with $l(\Delta)\sim N^{-1}$. Interpolating with the trivial $L^2$ result we get
$$\|(\prod_{i=1}^{\Lambda K}E_{R_i}^{(k)}g_i)^{\frac1{\Lambda K}}\|_{L^p(w_{B_N})}\lesssim_{\epsilon,K,p}N^{2(\frac12-\frac1p)+\epsilon}( \prod_{i=1}^{\Lambda K} (\sum_{l(\Delta)=N^{-1}}\|E_{\Delta}^{(k)}g_i\|_{L^p(w_{B_N})}^p)^{1/p})^{\frac1{\Lambda K}}$$
for each $2\le p\le n$.

By summing up over balls $B_N$ we also get the inequality
$$\|(\prod_{i=1}^{\Lambda K}E_{R_i}^{(k)}g_i)^{\frac1{\Lambda K}}\|_{L^p(w_{B_{N^k}})}\lesssim_{\epsilon,K,p}N^{2(\frac12-\frac1p)+\epsilon}( \prod_{i=1}^{\Lambda K} (\sum_{l(\Delta)=N^{-1}}\|E_{\Delta}^{(k)}g_i\|_{L^p(w_{B_{N^k}})}^p)^{1/p})^{\frac1{\Lambda K}},$$
for each ball $B_{N^k}$ with radius $N^k$. This can be read as $D_{k,K}(M,p)\lesssim_{\epsilon,K,p}M^{\frac2{k}(\frac12-\frac1p)+\epsilon}$ for $M\ge K$. The result now follows from Theorem \ref{ft2}.
\end{proof}

\bigskip

\section{The proof of Theorem \ref{tfek6}}
\label{last}
In this section we finish the proof of Theorem \ref{tfek6}. In fact we will prove the following more general result.
\begin{theorem}
\label{yjoui0tyiu90ty8y0-r5iy5690it}
If Conjecture \ref{ckconjh} holds true for some $k\ge 2$ then
\begin{equation}
\label{ji nvhfy vguitgurtgupyto}
D_{k}(N,p)\lesssim_{\epsilon}N^{\frac2k(\frac12-\frac1{p})+\epsilon}
\end{equation}
for each $2\le p\le k(k+3)-2$.
\end{theorem}
Recall  that we have verified Conjecture \ref{ckconjh} for $k=2,3$ in Section \ref{Trans}. In particular, Theorem \ref{tfek6} will follow.

\bigskip

 For the rest of this section fix $k\ge 2$. Let  $R_1,\ldots,R_{\Lambda K}$ be arbitrary distinct squares in $\Col_K$. Here and in the following,
$$\kappa_p=\frac{p-n}{p-2},\;\;n=\frac{k(k+3)}{2}.$$

Corollary \ref{gjityiophjytophpotigirti0-we=fdcwee=w=} and the H\"older inequality imply that  for each $g_i:R_i\to\C$,  each $l\ge 2$ and each $\kappa_p\le \kappa\le 1$ we have
$$\|(\prod_{i=1}^{\Lambda K}\sum_{\atop{l(\tau)=N^{-2^{-l}}}}|E_{\tau}^{(k)}g_i|^2)^{\frac1{2\Lambda K}}\|_{L^{p}(w_{B_N})} \le C_{p, K,\epsilon}N^{\frac{\kappa}{2^{l-1}}(\frac12-\frac1p)+\epsilon}\times $$
\begin{equation}
\label{fe21}
 \|(\prod_{i=1}^{\Lambda K}\sum_{\atop{l(\Delta)=N^{-2^{-l+1}}}}
|E_{\Delta}^{(k)}g_i|^2)^{\frac1{2\Lambda K}}\|_{L^{p}(w_{B_N})}^{1-\kappa}(\prod_{i=1}^{\Lambda K}\sum_{\atop{l(\tau)=N^{-2^{-l}}}}
\|E_{\tau}^{(k)}g_i\|_{L^{p}(w_{B_N})}^p)^{\frac{\kappa}{\Lambda Kp}}.
\end{equation}
The value $\kappa=\kappa_p$ suffices for proving \eqref{ji nvhfy vguitgurtgupyto} at the endpoint $p=k(k+3)-2$, while use of $\kappa>\kappa_p$ will be made in order to prove \eqref{ji nvhfy vguitgurtgupyto} for $2\le p<k(k+3)-2$. While the conjectured values of $D_k(N,p)$ exhibit an affine dependence  on $\frac1p$,  we are not aware of any interpolation argument when $k\ge 3$, that recovers \eqref{ji nvhfy vguitgurtgupyto} for $2< p<k(k+3)-2$, from the correct estimates for $D_k(N,2)$ and $D_k(N,k(k+3)-2)$. This is because we decompose into curved regions that, when $k\ge 3$, are no longer straight tubes. There is however an interpolation available for $k=2$, see for example \cite{BD3}.

\bigskip

We will find useful the following immediate consequence of the Cauchy--Schwartz inequality. While the exponent $2^{-s}$ in $N^{2^{-s}}$ can be improved by making use of   transversality, the following trivial estimate will suffice for our purposes.
\begin{lemma}
\label{wlem0081}Consider $\Lambda K$  squares $R_1,\ldots,R_{\Lambda K}\in \Col_K$. Assume $g_i$ is supported on $R_i$.
Then for $1\le p\le\infty$ and $s\ge 2$
$$\|(\prod_{i=1}^{\Lambda K}|E_{R_i}^{(k)}g_i|)^{\frac1{\Lambda K}}\|_{L^{p}({w_{B_N}})}\le N^{2^{-s}}\|(\prod_{i=1}^{\Lambda K}\sum_{\atop{l(\tau_s)=N^{-2^{-s}}}}|E_{\tau_s}^{(k)}g_i|^2)^{\frac1{2\Lambda K}}\|_{L^{p}(w_{B_N})}.$$
\end{lemma}

\bigskip

\begin{proof}[of Theorem \ref{yjoui0tyiu90ty8y0-r5iy5690it}]

Fix $k\ge 2$ so that  Conjecture \ref{ckconjh} holds true. Because of \eqref{ji nvhfy vguitgurtgupytooihnopyin9pnyiup97iui9076}, we can restrict attention to $n<p<k(k+3)-2$.

Fix $\epsilon>0$, $K\ge 2$, to be chosen later.
\bigskip

Let $R_1,\ldots, R_{\Lambda K}\in \Col_K$ be arbitrary squares and assume $g_i$ is supported on $R_i$. Define $m$ to be the smallest integer so that $2^{-m}\le \frac1k$.

Start with Lemma \ref{wlem0081}, continue with iterating \eqref{fe21} $s-m+1$ times, and invoke \eqref{fe22} at each step to write for each $p>n$ and each $\kappa_p\le \kappa\le 1$
$$\|(\prod_{i=1}^{\Lambda K}|E_{R_i}^{(k)}g_i|)^{\frac1{\Lambda K}}\|_{L^{p}({B_N})}\le  N^{2^{-s}}\|(\prod_{i=1}^{\Lambda K}\sum_{\atop{l(\tau_s)=N^{-2^{-s}}}}|E_{\tau_s}^{(k)}g_i|^2)^{\frac1{2\Lambda K}}\|_{L^{p}(w_{B_N})}\le $$
$$ N^{2^{-s}}(C_{p,K,\epsilon}N^\epsilon)^{s-m+1}\prod_{l=m}^s
N^{\frac{\kappa}{2^{l-1}}(1-\kappa)^{s-l}(\frac12-\frac1p)}\times$$$$\|(\prod_{i=1}^{\Lambda K}\sum_{\atop{l(\tau)=N^{-2^{1-m}}}}|E_{\tau}^{(k)}g_i|^2)^{\frac1{2\Lambda K}}\|_{L^{p}(w_{B_N})}^{(1-\kappa)^{s-m+1}}\times $$
$$\prod_{i=1}^{\Lambda K}\left[\prod_{l=m}^s(\sum_{\atop{l(\tau)=N^{-2^{-l}}}}\|E_{\tau}^{(k)}g_i\|_{L^{p}(w_{B_N})}^p)^{\frac{\kappa}{p}(1-\kappa)^{s-l}}\right]^{\frac1{\Lambda K}}$$

$$ \le N^{2^{-s}}(C_{p,K,\epsilon}N^\epsilon)^{s-m+1}(\prod_{i=1}^{\Lambda K}\sum_{\atop{l(\Delta)=N^{-1/k}}}
\|E_{\Delta}^{(k)}g_i\|_{L^{p}(w_{B_N})}^p)^{\frac{1-{(1-\kappa)^{s-m+1}}}{\Lambda Kp}}\times$$
$$N^{2^{1-s}\kappa(\frac12-\frac1p)\frac{1-[2(1-\kappa)]^{s-m+1}}{1-2(1-\kappa)}}\times\|(\prod_{i=1}^{\Lambda K}\sum_{\atop{l(\tau)=N^{-2^{1-m}}}}|E_{\tau}^{(k)}g_i|^2)^{\frac1{2\Lambda K}}\|_{L^{p}(w_{B_N})}^{(1-\kappa)^{s-m+1}}\times$$
\begin{equation}
\label{fe34}
 D_k(N^{1-k2^{-s}},p)^{\kappa}D_k(N^{1-k2^{-s+1}},p)^{\kappa(1-\kappa)}\cdot\ldots\cdot D_k(N^{1-k2^{-m}},p)^{\kappa(1-\kappa)^{s-m}}.
\end{equation}

Note that the inequality
$$\|(\sum_{\atop{l(\Delta)=N^{-2^{1-m}}}}|E_{\Delta}^{(k)}g_i|^2)^{\frac1{2}}\|_{L^{p}(w_{B_N})}\le N^{O_p(1)}(\sum_{\atop{l(\Delta)=N^{-1/k}}}\|E_{\Delta}^{(k)}g_i\|_{L^{p}(w_{B_N})}^p)^{1/p}$$
is a consequence of Minkowski's inequality and standard truncation arguments. The precise value of the exponent is not relevant, all that matters is that it is $O_p(1)$. Applying H\"older's inequality leads to

$$\|\prod_{i=1}^{\Lambda K}(\sum_{\atop{l(\Delta)=N^{-2^{1-m}}}}|E_{\Delta}^{(k)}g_i|^2)^{\frac1{2\Lambda K}}\|_{L^{p}(w_{B_N})}\le N^{O_p(1)}(\prod_{i=1}^{\Lambda K}\sum_{\atop{l(\Delta)=N^{-1/k}}}\|E_{\Delta}^{(k)}g_i\|_{L^{p}(w_{B_N})}^p)^{\frac{1}{\Lambda Kp}}.$$

Using this and maximizing over all choices of $R_i$, \eqref{fe34} has the following consequence, for all $N\ge K$
$$D_{k,K}(N,p)\le (C_{p,K,\epsilon}N^\epsilon)^{s-1} N^{2^{-s}}N^{2^{1-s}\kappa(\frac12-\frac1p)\frac{1-[2(1-\kappa)]^{s-m+1}}{1-2(1-\kappa)}}\times$$
\begin{equation}
\label{fe23}
  D_k(N^{1-k2^{-s}},p)^{\kappa}D_k(N^{1-k2^{-s+1}},p)^{\kappa(1-\kappa)}\cdot\ldots\cdot D_k(N^{1-k2^{-m}},p)^{\kappa(1-\kappa)^{s-m}}  N^{O_p((1-\kappa)^s)}.
\end{equation}
\bigskip

Let $\gamma_p$ be the unique positive number such that
$$\lim_{N\to\infty}\frac{D_k(N,p)}{N^{\gamma_p+\delta}}=0,\;\text{for each }\delta>0$$
and
\begin{equation}
\label{fe27}
\limsup_{N\to\infty}\frac{D_k(N,p)}{N^{\gamma_p-\delta}}=\infty,\;\text{for each }\delta>0.
\end{equation}
The existence of such $\gamma_p$ is guaranteed by \eqref{fe3008}.
By using the fact that $D_k(N,p)\lesssim_\delta N^{\gamma_p+\delta}$ in \eqref{fe23}, it follows that for each $\delta,\epsilon>0$ and $K, s\ge 2$
\begin{equation}
\label{fe32}
\limsup_{N\to\infty}\frac{D_{k,K}(N,p)}{N^{\gamma_{p,\delta,s,\epsilon,\kappa}}}<\infty
\end{equation}
 where
$$\gamma_{p,\delta,s,\epsilon,\kappa}=\epsilon(s-1)+2^{-s}+\kappa(\gamma_p+\delta)(\frac{1-(1-\kappa)^{s-m+1}}{\kappa}-k2^{-s}
\frac{1-[2(1-\kappa)]^{s-m+1}}{2\kappa-1})+$$
\begin{equation}
\label{nmcvuyfgurivgnyy35t789t89y890}
2^{1-s}\kappa(\frac12-\frac1p)\frac{1-[2(1-\kappa)]^{s-m+1}}{2\kappa-1} +O_p((1-\kappa)^s).
\end{equation}
Recall that our goal is to prove that for each $n< p\le k(k+3)-2$
\begin{equation}
\label{djcfhughyvg7trgyn7otgy7nn845nyon}
\gamma_{p}\le \frac2k(\frac12-\frac1{p}).
\end{equation}
\bigskip

Assume for contradiction that this is not true, for some $p$. Then, for $\kappa$ larger than but close enough to $\frac12$  we have
\begin{equation}
\label{fe26}
\gamma_p>\frac2{k}(\frac{2\kappa-1}{2\kappa}+\frac12-\frac1p).
\end{equation}
Note that \eqref{fe32} holds for this $\kappa$, as $\kappa_p\le \frac12<\kappa$. A simple computation using that $2(1-\kappa)<1$ and \eqref{fe26} shows that for $s$ large enough and for $\epsilon,\delta$ small enough we have
\begin{equation}
\label{fe33}
\gamma_{p,\delta,s,\epsilon,\kappa}<\gamma_p.
\end{equation}
This follows by noticing that \eqref{nmcvuyfgurivgnyy35t789t89y890} implies
$$2^s(\gamma_{p,\delta,s,\epsilon,\kappa}-\gamma_p)=o_{\epsilon,\delta,s}(1)+1+\frac{2\kappa}{2\kappa-1}(\frac12-\frac1p-\gamma_p\frac{k}{2}).$$
Fix such  $\epsilon,\delta, s,\kappa$.

Now, \eqref{fe31}  shows that for $N\ge K$
\begin{equation}
\label{fek16}
D_k(N,p)\le N^{\beta(K,p)+\frac2k(\frac12-\frac1p)}+\Omega_{K,p}\log_K N\max_{1\le M\le N}\left[(\frac{M}{N})^{\frac2k(\frac1p-\frac12)}D_{k,K}(M,p)\right].
\end{equation}
We  argue that $\gamma_{p,\delta,s,\epsilon,\kappa}\le \frac{2}{k}(\frac12-\frac1p)$. If this were not true, we could choose $K$ large enough so that $$\beta(K,p)+\frac2k(\frac12-\frac1p)\le\gamma_{p,\delta,s,\epsilon,\kappa}.$$
Combining this with \eqref{fe32} and \eqref{fek16} leads to
$$D_k(N,p)\le(\Omega_{K,p}\log_K N+1)N^{\gamma_{p,\delta,s,\epsilon,\kappa}}.$$
This of course contradicts \eqref{fe33} and \eqref{fe27}.

Using now that $\gamma_{p,\delta,s,\epsilon,\kappa}\le \frac{2}{k}(\frac12-\frac1p)$ together with \eqref{fe32}, we can rewrite  \eqref{fek16} as follows
$$D_k(N,p)\le N^{\beta(K,p)+\frac2k(\frac12-\frac1p)}+\Omega_{K,p}\log_K NN^{\frac2k(\frac12-\frac1p)}.$$
By choosing $K$ as large as needed and using the definition of $\gamma_p$, this forces $\gamma_p\le \frac2k(\frac12-\frac1p)$, contradicting our original assumption that \eqref{djcfhughyvg7trgyn7otgy7nn845nyon} is false.
\end{proof}

\end{document}